\documentclass[12pt]{amsart}
\usepackage[utf8]{inputenc}

\usepackage[english]{babel}
\usepackage{amsmath, amssymb, amsthm, amscd,color,comment}
\usepackage{cancel}
\usepackage{cite}
\usepackage{alltt}
\usepackage[dvipsnames]{xcolor}
\usepackage{array}
\usepackage{caption} 
\usepackage{mathtools}

\usepackage{enumerate}
\usepackage{cancel}

\input xy
\xyoption{all}

\newcommand{\Q}{\mathbb{Q}}

\newcommand{\ov}{\overline}

\newcommand{\F}{\mathcal{F}}

\newcommand{\fq}{\mathbb{F}_q}
\newcommand{\kk}{\mathbb{K}}

\newcommand{\fqc}{\overline{\mathbb{F}}_q}
\newcommand{\con}{\operatorname{Con}}
\newcommand{\End}{\operatorname{End}}

\newcommand{\fqs}{\mathbb{F}_{q^2}}

\newcommand{\Z}{\mathbb Z}

\newcommand{\aut}{\operatorname{Aut}}
\newcommand{\gal}{\operatorname{Gal}}
\newcommand{\PGL}{\operatorname{PGL}}

\newcommand{\divi}{\text{div}}

\newcommand{\car}{\text{char}}

\theoremstyle{plain}
\newtheorem{thm}{Theorem}[section]

\newtheorem{prop}[thm]{Proposition}
\newtheorem{lem}[thm]{Lemma}
\newtheorem{cor}[thm]{Corollary}
\newtheorem{rem}[thm]{Remark}

\usepackage{cancel}
\usepackage[dvipsnames]{xcolor}

\begin{document}
\title{Cyclotomic function fields over finite fields with irreducible quadratic modulus}



\author{Nazar Arakelian and Luciane Quoos}

\address{Centro de Matemática, Computação e Cognição, Universidade Federal do ABC, Avenida dos Estados 5001, CEP 09210-580, Santo André, SP, Brazil}
\email{n.arakelian@ufabc.edu.br}

\address{Instituto de Matemática, Universidade Federal do Rio de Janeiro, Cidade Universitária,
	CEP 21941-909, Rio de Janeiro, RJ, Brazil}
\email{luciane@im.ufrj.br}


\begin{abstract} 
Let $\fq$ be the finite field of order $q$ and $F=\fq(x)$ the rational function field. In this paper, we give a characterization of the cyclotomic function fields $F(\Lambda_M)$ with modulus $M$, where $M \in \fq[T]$ is a monic and irreducible polynomial of degree two. More precisely, we show that $F(\Lambda_M)$ is the only function field, up to $\fq$-isomorphism, with $q+1$ $\fq$-rational places, genus $(q+1)(q-2)/2$ and a subgroup of automorphisms over $\fq$  isomorphic to $\fqs^*$. We also provide the full automorphism group of $F(\Lambda_M)$ in odd characteristic, extending results of \cite{MXY2016} where the automorphism group of $F(\Lambda_M)$ over $\fq$ was computed.
\end{abstract}

\maketitle

\section{Introduction}

In the classification of algebraic function fields (of one variable), which is a leading problem in algebraic geometry, the study of their birational invariants, for instance, their automorphism group, genus and number of rational places, is essential. Although one cannot expect to classify algebraic function fields with only one of such invariants, there are some cases in which the classification is possible in terms of more than one of them. To illustrate this, we have for example the classification of the so called maximal function fields with high genus. Recall that, if $\fqs$ denotes a finite field with $q^2$ elements, a function field over $\fqs$  of genus $g$ is said to be maximal if the number $N$ of its $\fqs$-rational places meets the Hasse-Weil bound, that is
$$
N=q^2+1+2gq.
$$
It is a well know result that a maximal function field over $\fqs$ has genus less or equal than $q(q-1)/2$.
 It is shown in \cite{RS1994} that the Hermitian function field, defined by $\fqs(x,y)$ with $y^q+y=x^{q+1}$ is the only maximal function field over $\fqs$ of genus $q(q-1)/2$ up to $\fqs$-isomorphism.  The classification of  maximal function fields with the second largest genus $(q-1)^2/4$ in odd characteristic can be found in \cite{FGT1997},  $\fqs(x,y)$ with $y^q+y=x^{\frac{q+1}{2}}$. In even characteristic the result can be found in \cite{AT1999} and the function field is given by $\fqs(x,y)$ where $y^{q/2}+\cdots +y^2+y=x^{q+1}. $

In some situations, the characterization is obtained via the genus and the some subgroup of the automorphism group. For instance, in \cite{AK2015} the Artin-Mumford function field is obtained as the only function field  (up to $\mathbb{F}_p$-isomorphism) over $\mathbb{F}_p$ of genus $(p-1)^2$ with a subgroup of automorphisms isomorphic to $(\Z_p \times \Z_p) \rtimes D_{p-1}$. In turn, there are function fields that are characterized via an automorphism subgroup and the structure of some fixed fields, see  \cite{AS2017} for example.

However, in several cases, more than two invariants are used in the classification.
For $q=n^3$ in \cite[Theorem 8]{GK2009} the authors classified the $GK$ algebraic function field over $\mathbb{F}_{q^2}$  as the only, up to $\fqs$-isomorphism, $\fqs$-maximal function field of genus $g=\tfrac 12 (n^3+1)(n^2-2)+1$ with an $\mathbb{F}_{q^2}$ automorphism group of order $n^3(n^3+1)(n^2-1)(n^2-n+1)$. In \cite{TT2019}, it is given a characterization of the so called Ree function field via the number of rational places, the genus, and the shape of two elements of the Weierstrass semigroup at a rational place.  

One characteristic that the aforementioned algebraic function fields have in common is that they are of great importance in the literature, both in theory and for applications. In line with this spirit, we highlight here the cyclotomic function fields over finite fields. Such function fields arose from the parallel between the number fields, i.e., finite extensions of the field $\Q$ of rational numbers, and finite extensions of the rational function field $\fq(x)$ over $\fq$. More precisely, Hayes in \cite{Hayes1974} developed an explicit Class Field Theory for $\fq(x)$, in the sense that he constructed the maximal abelian extension of $\fq(x)$ by means of certain function fields; since any abelian extension of $\Q$ is a subfield of a cyclotomic number field $\Q(\zeta)$ for a suitable $n$-th primitive root of unity $\zeta$ (Kronecker-Weber Theorem), these algebraic function fields introduced by Hayes are know as cyclotomic function fields over finite fields. 
 
In recent years one can find interesting research about cyclotomic function fields over finite fields, see \cite{KW2021, YL2021} and \cite{LL2020} and references in there. In addition to being important by their own nature, cyclotomic function fields have being a valuable tool for applications. For instance, for binary sequences with a low correlation in \cite{JMX2022}, and construction of sequences with high nonlinear complexity, see \cite{LXY2017}. 

A cyclotomic function field is determined by a polynomial $M \in \fq[x]$, where $x$ is an indeterminate over $\fq$, and in this case we denote such field by $F(\Lambda_M)$, where $F=\fq(x)$; the polynomial $M$ is said to be the modulus of $F(\Lambda_M)$. When $M$ is monic of degree one, then the genus of $F(\Lambda_M)$ is $0$, that is, $F(\Lambda_M)$ is a rational function field. Recently in \cite{MXY2016}, the case in which $M$ is a monic and irreducible quadratic polynomial was studied in details. In particular,  from \cite{Hayes1974}, the following properties hold in this case:
\begin{itemize}
\item $\aut_{\fq}(F(\Lambda_M))$ has a subgroup $G$ isomorphic to $\fqs^{*}$.\\
\item $g(F(\Lambda_M))=\frac{(q+1)(q-2)}{2}$.\\
\item $F(\Lambda_M)$ has exactly $q+1$ rational places over $\fq$.
\end{itemize}

One of the results achieved in \cite[Theorem 4.6]{MXY2016} was that when $M$ is monic, quadratic and irreducible, then the automorphism group $\aut_{\fq}(F(\Lambda_M))$ equals $\gal(F(\Lambda_M)/F)$ which is isomorphic to $\fqs^{*}$ by \cite[Theorem 2.3]{Hayes1974}.

The aim of this work is twofold. Firstly, we prove that the properties listed above completely characterize $F(\Lambda_M)$ for $M$ monic, quadratic and irreducible. This will be done in Section \ref{character}. Then, in Section \ref{full} we compute the full automorphism group of $F(\Lambda_M)$ in odd characteristic, that is, the group $\aut_{\fqc}(\fqc F(\Lambda_M))$. This extends the result of \cite[Theorem 4.6]{MXY2016}. 

\section{Preliminaries}

In this section we fix some notation and basic results of the theory of automorphism groups of algebraic function fields and cyclotomic function fields. From now on, $\fq$ will always denote the finite field with $q$ elements, where $q$ is a power of a prime number $p$, and $\fqc$ will denote the algebraic closure of $\fq$. We point out that some of the results presented here are rephrased to suit well our context.

\subsection{Automorphims of function fields}

Our notation and terminology are standard. For an exhaustive treatise of the theory of algebraic curves
and algebraic function fields, the reader is referred to \cite{HKT}, \cite{VSbook} and \cite{St}. Let $\F$ be an algebraic function field (of one variable) over $\fqc$ and  $H$ a subgroup of the automorphism group $\aut_{\fqc}(\F)$ of $\F$. We consider the action of $H$ on the places of $\F$ and for a fixed place $P$ of $\F$ we let $H_P=\{ h \in H \, :\, h(P)=P\}$ denote the stabilizer of $P$, and $H(P)=\{ h(P) \, : \, h \in H\}$ be the orbit of $P$ under the action.

We say that the orbit of an automorphism group is {\it tame} if $p$ does not divide the order of the stabilizer of its places. Otherwise, we say that the orbit is {\it non-tame}. In a similar way we say that a group $G$ is {\it tame} if $p$ does not divide the order of $G$ and {\it non-tame} otherwise.

Let $S \leq \aut_{\fqc}(\F)$ and denote by $\F^{S}$ the subfield of $\F$ fixed by $S$. It is well known that the extension $\F / \F^S$ is Galois with Galois group $S=\gal(\F / \F^S)$. Given a place $P$ of $\F$ and a place $\tilde{P}$ of $\F^{S}$ lying under $P$, the ramification index  $e(P|\tilde{P})$ coincides with the size of the stabilizer of $P$ in S, that is, $e(P|\tilde{P})=|S_P|$. If $g$ and $g^{\prime}$ denote the genus of $\F$ and $\F^S$, respectively, and $\ell_1,\ldots,\ell_k$ denote the sizes of the short orbits of $S$, the Riemann-Hurwitz genus formula is
\begin{equation}\label{rhf}
2g-2 \geq |S|(2g^{\prime}-2)+\sum_{i=1}^{k}(|S|-\ell_i),
\end{equation}
and equality holds if $\gcd(p,|S|)=1$; see \cite[Theorem 11.57 and Remark 11.61]{HKT}.

The following results will be systematically used in the sequel of the paper.

\begin{thm}\cite[Lemma 11.44]{HKT}\label{stab}
Let $\F$ be an algebraic function field over $\fqc$  and $G_P \leq \aut_{\fqc}(\F)$ be the stabilizer of a place $P$ of $\F$. Then the $p$-Sylow subgroup $S_p$ of $G_P$ is a normal subgroup and the quotient group $G_P/S_p$ is cyclic of order prime to $p$. 
\end{thm}

\begin{thm}\cite[Theorem 11.79]{HKT}\label{orderab}
Let $\F$ be an algebraic function field  of genus $g \geq 2$ and $G$ an abelian $\fqc$-automorphism subgroup of $\aut_{\fqc}(\F)$.Then 
$$
|G| \leq 
\begin{cases}
4g+ 4, \quad \text{if } p \neq 2,\\
4g+ 2, \quad \text{if } p= 2.
\end{cases} 
$$
\end{thm}

\begin{thm}\cite[Theorem 11.56]{HKT}\label{orbits}
Let $\F$ be an an algebraic function field  of genus $g \geq 2$ and $G \leq \aut_{\fqc}(\F)$. Suppose that $|G|>84(g-1)$. Then the fixed field $\F^{G}$ is rational and one  of the following holds: 
\begin{itemize}
\item[(a)] $G$ has exactly three short orbits, two tame and one non-tame.
\item[(b)] $G$ has exactly two short orbits, both non-tame.
\item[(c)] $G$ has exactly one short orbit, which is non-tame.
\item[(d)] $G$ has exactly two short orbits, one tame and one non-tame.
\end{itemize}
\end{thm}

Given a function field $\F/\fq$, it is known that the automorphisms of $\F$ can be extended to automorphisms of $\fqc \F/\fqc$ (the constant field extension of $\F/\fq$). In fact, there exists a group monomorphism $\aut_{\fq}(\F)\hookrightarrow \aut_{\fqc}(\fqc \F)$, and this inclusion respects the action on the places of the respective function fields, see e.g. \cite[Corollary 14.3.9]{VSbook}. The group $\aut_{\fqc}(\fqc \F)$ is the full automorphism group of $\F$. For simplicity, in what follows, we will denote such group by $\aut_{\fqc}(\F)$.

\subsection{Cyclotomic function fields}

In this subsection, we give a brief introduction to the theory of cyclotomic function fields. All the results can be found in \cite{Hayes1974} and, in more details, in \cite[Chapter 12]{VSbook}. 

For a transcendent element $x$ over $\fq$, set $R=\fq[x]$ the polynomial ring on $x$ with coefficients in $\fq$ and $F=\fq(x)$ its field of fractions. Fix an algebraic closure $\ov{F}$ of $F$ and denote by $\End_{\fq}(\ov{F})$ the ring of endomorphisms of $\ov{F}$ (here, $\ov{F}$ is seen as an $\fq$-vector space). Consider $\varphi \in \End_{\fq}(\ov{F})$ given by
$$
\varphi(z)=z^q+xz, \ \ z \in \ov{F}.
$$
Thus we can define a ring homomorphism
\begin{eqnarray}
R & \longrightarrow & \End_{\fq}(\ov{F}) \nonumber \\
f(x) & \mapsto & f(\varphi). \nonumber
\end{eqnarray}
This map induces a structure of $R$-module in $\ov{F}$ via the action $z^{f(x)}=f(\varphi)(z)$ of $R$ in $\ov{F}$. More precisely, given $z \in \ov{F}$ and $f(x)=\sum_{i=0}^{n}a_ix^i \in R$, we have
$$
z^{f(x)}=\sum_{i=0}^{n}a_i\varphi^i(z),
$$
where $\varphi^0(z)=z$  and $\varphi^k=\varphi \circ \varphi^{k-1}$ for $k>1$. Given $M \in R$ with $M \neq 0$, the set of $M$-torsion points
$$
\Lambda_M=\{z \in \ov{F} \ | \ z^M=0\}
$$
is a finite $R$-submodule of $\ov{F}$. As a matter of fact, $|\Lambda_M|=q^{\deg M}$ and $\Lambda_M \cong R/\langle M \rangle $. The cyclotomic function field with modulus $M$, denoted by $F(\Lambda_M)$, is defined by the subfield of $\ov{F}$ generated over $F$ by the elements of $\Lambda_M$. 

Furthermore, $F(\Lambda_M)$ is a Galois extension of $F$ with Galois group 
$$
\gal(F(\Lambda_M)/F) \cong (R / \langle M \rangle)^{\times}, 
$$
where $(R / \langle M \rangle)^{\times}$ is the unit group of $R / \langle M \rangle$. In particular, if $M$ is monic and irreducible of degree $d$, one has $\gal(F(\Lambda_M)/F) \cong \mathbb{F}_{q^d}^{*}$. In this case, it follows from \cite[Corollary 4.2]{Hayes1974} that the genus of $F(\Lambda_M)$ is
$$
g(F(\Lambda_M))=\frac{1}{2} \left(\frac{((d-1)q-d)(q^d-1)}{q-1} -d+2\right).
$$
 
\section{The characterization of $F(\Lambda_M)$}\label{character}

Let $F=\fq(x)$ be the rational function field.
The goal of this section is to characterize the cyclotomic function fields $F(\Lambda_M)$ with modulus $M$, where $M \in \fq[T]$ is a monic and irreducible polynomial of degree two. This result is achieved in Theorem \ref{gencyc}. 

The following result can be obtained by a variation of the discussion at the beginning of \cite[Section 4]{MXY2016}.

\begin{prop}\label{pr}
Let $\fq$ be the finite field of order $q$. Assume that $f(T)=T^2+aT+b \in \fq[T]$ is irreducible. The cyclotomic function field $F(\Lambda_M)$ with modulus $M=f(T)$ is $\fq$-isomorphic to the function field $\fq(v,y)$ defined over $\fq$ by 
$$
y^{q-1}=  -\frac{\gamma v^2+av+(b/\gamma)}{v^q-v}
$$
for all $\gamma \in \fq^{*}$.
\end{prop}
\begin{proof}
First, it is not difficult to check that $f(T)$ is irreducible over $\fq$ if, and only if, $\gamma T^2+aT+(b/\gamma)$ is irreducible over $\fq$. By the construction of cyclotomic function fields, one can see that $F(\Lambda_M)=\fq(x,y)$, with 
\begin{equation} \label{eq0}
y^{q^2-1}+(x^q+x+a)y^{q-1}+x^2+ax+b=0.
\end{equation}
For $\gamma \in \fq^*$, set $v=\gamma^{-1}(x+y^{q-1})$. Then
\begin{align*}
\gamma y^{q-1}(v^q-v)&=y^{q-1}(x^q+y^{q^2-q}-x-y^{q-1})\\
&=y^{q^2-1}-y^{2q-2}+(x^q-x)y^{q-1}\\
&=-y^{2q-2}-(2x+a)y^{q-1}-x^2-ax-b\\
&=-(x+y^{q-1})^2-a(x+y^{q-1})-b\\
&=-(\gamma v)^2-a\gamma v-b
\end{align*}
And the result follows.
\end{proof}

In what follows, $\F$  denotes a function field over $\fq$ of genus $g(\F)$ satisfying the following three properties:
\begin{itemize}
\item $\aut_{\fq}(\F)$ has a subgroup $G$ isomorphic to $\fqs^{*}$.\\
\item $g(\F)=\frac{(q+1)(q-2)}{2}$, and\\
\item $\F$ has exactly $q+1$ rational places over $\fq$.
\end{itemize}

\begin{lem}\label{lemma1}
The action of $G$ on the set of places of $\F$ has two short orbits $\Omega_1$ and $\Omega_2$, where $\Omega_1$ is the set of $\fq $-rational places of $\F$ and $\Omega_2$ has only one place of degree two over $\fq$. In particular, if $H\leq G$ is the only subgroup of order $q-1$, then $\F^H$ is a rational function field.
\end{lem}
\begin{proof}
Let us consider $\F^{\prime}=\fqc \F$ as a function field defined over $\fqc$, that is, $\F^{\prime}/\fqc$ is a constant field extension of $\F/\fq$. Since $\F^{\prime}/(\F^{\prime})^G$ is tame, the Riemann-Hurwitz Formula applied to the extension $\F^{\prime}/(\F^{\prime})^G$ gives
\begin{equation}\label{eq1}
q^2-q-4=(2\bar{g}-2)(q^2-1)+\sum_{i=1}^{k}(q^2-1-\ell_i),
\end{equation}
where $\ell_1,\ldots,\ell_k$ are the sizes of the short orbits of $G$ in the set of places of $\F^{\prime}$ and $\bar{g}$ is the genus of $(\F^{\prime})^G$. 
From \eqref{eq1}, we clearly have $\bar{g} \leq 1$. If $\bar{g}=1$, then
$$
\sum_{i=1}^{k}\ell_i=(k-1)(q^2-1)+q+3.
$$
Since $\ell_i$ is a proper divisor of $q^2-1$, the last equality gives $k=1$, which leads to $\ell_1=q+3$.  However, $G$ is defined over $\fq$ and then the set formed by the $q+1$ $\fq$-rational places of $\F$ must be a union of orbits of $G$, and thus we obtain a contradiction. 

Thus $\bar{g}=0$, and \eqref{eq1} reads
\begin{equation}\label{eq2}
\sum_{i=1}^k \ell_i=(k-3)(q^2-1)+q+3,
\end{equation}
with $k \geq 3$. Recall that $\ell_i \leq (q^2-1)/2$ for all $i$. Again, $G$ is defined over $\fq$, whence the set of $\fq$-rational places of $\F$ is preserved by $G$. In particular, the set of rational places is a union of orbits of $G$. Thus $\ell_i \leq q+1$ for at least one $i$. Using this in \eqref{eq2} gives $k \leq 4$.

Without loss of generality suppose $\ell_1 \leq  q+1$. If $\ell_1 <  q+1$, then there is at least one second orbit of size less than $q+1$, say $\ell_2 <q+1$. Then equation \eqref{eq2} allows to conclude $k\leq 3$, and this yields $k=3$.
Now we suppose $\ell_1=q+1$. We want to prove that $k=3$.  Assume that $k=4$. Then  \eqref{eq2} gives 
\begin{equation}\label{eq3}
\ell_2+\ell_3+\ell_4=q^2+1.
\end{equation}
We may suppose $\ell_2 \leq \ell_3 \leq \ell_4$.
 
If $q$ is even, since $\ell_i\leq (q^2-1)/3$ for $i \in \{2,3,4\}$, then $q^2+1=\ell_2+\ell_3+\ell_4\leq q^2-1$, a contradiction. Now suppose that $q$ is odd. Note that if $\ell_i < (q^2-1)/2$, then $\ell_i  \leq  (q^2-1)/3$. From \eqref{eq3}, we have without loss of generality that  $\ell_4=(q^2-1)/2$. Then $\ell_2+\ell_3=\frac{q^2+3}2$. Furthermore, we have $\ell_2, \ell_3 \leq \frac{q^2-1}3$. Indeed, if w.l.o.g. $\ell_3=(q^2-1)/2$, then \eqref{eq3} provides $\ell_2=2$. Hence, every place in the union of the four short orbits of $G$ have an involution in its stabilizer, and since $G$ is cyclic, such involution is the same for all place. Thus, the Riemann-Hurwitz Formula applied to the fixed field of such involution lead us to a contradiction.

 Let $\ell_i=\frac{q^2-1}{a_i},$ where $ 3 \leq a_i $ is a divisor of $q^2-1$.  We obtain 
$$\frac{q^2-1}{a_2}+\frac{q^2-1}{a_3}=\frac{q^2+3}2 \Rightarrow 2(a_2+a_3)(q^2-1)=a_2a_3(q^2+3) \Rightarrow 2(a_2+a_3)> a_2a_3.$$
Since $a_2, a_3 \geq 3$ we get $a_2, a_3 \in \{3,4, 5\}$, and  w.l.o.g. we obtain $a_2=3, a_3=4$ and $q=5$. Thus, if $q \neq 5$, we must have $k=3$. 
In the case $q=5$ we can have four orbits of size $6, 6, 8$ and $12$, but this case can be easily ruled out by the application of the Riemann-Hurwitz Formula to the fixed field of the unique involution of $G$. We also conclude $k=3$ for $q=5$.

Therefore we have three orbits $\Omega_1$, $\Omega_2$ and $\Omega_3$ of respective sizes $\ell_1$, $\ell_2$ and $\ell_3$, where $\ell_1+\ell_2+\ell_3=q+3$. This rises the next possibilities (up two relabelling the $\ell_i^\prime$s):
\begin{itemize}
\item[(1)] $\ell_1+\ell_2=q+1$ and $\ell_3=2$.
\item[(2)] $\ell_1=q+1$ and $\ell_2=\ell_3=1$.
\end{itemize}

We aim to show that case (1) above does not hold. Since the size of an orbit must divide $q^2-1$, then if $q$ is even (1) can be immediately  ruled out. Assume that $q$ is odd. If $\ell_1=\ell_2$, we have that $\ell_1=\ell_2=(q+1)/2$. Denote by $N$ the subgroup of $G$ of order $2(q-1)$. The Riemann-Hurwitz Formula applied to the extension $\F^{\prime}/(\F^{\prime})^N$ gives
$$
q^2-q-4=(2g((\F^{\prime})^{N})-2)2(q-1)+(q+1)(2(q-1)-1)+2(2(q-1)-s),
$$
where $s=1$ if $q \equiv 3 \mod 4$ and $s=2$ if $q \equiv 1 \mod 4$. Both cases lead to a contradiction.

Hence, suppose w.l.o.g. $\ell_2>\ell_1$. We claim that that $\ell_1$ and $\ell_2$ have a common non-trivial divisor.  Indeed, assume that $\gcd(\ell_1,\ell_2)=1$. Then there are integers $c_1,c_2$ such that $c_1 \ell_1+c_2\ell_2=1$, thus
$$
(q+1)c_1=(\ell_1+\ell_2)c_1=1+\ell_2(c_1-c_2),
$$
which gives $\gcd(\ell_2,q+1)=1$. In the same way, we obtain $\gcd(\ell_1,q+1)=1$. Since both $\ell_1, \ell_2$ divide $q^2-1$, we conclude that both are divisors of $q-1$. If $\ell_2=q-1$, then we must have $\ell_1=2$ as $\ell_1+\ell_2=q+1$, and hence $2$ is a common divisor of $\ell_1$ and $\ell_2$, a contradiction.  If both are proper divisors of $q-1$, then
$$
q+1=\ell_1+\ell_2 \leq \frac{q-1}{2}+\frac{q-1}{2}=q-1,
$$
again a contradiction. This means that there is a prime number $\ell$ dividing  $\gcd(\ell_1,\ell_2)$. Let $E$ be the subgroup of $G$ with $|E|=(q^2-1)/\ell$. Note that for $i=1,2$, we have
$$
\frac{q^2-1}{\ell}=|G_{P_i}|\cdot \left(\frac{\ell_i}{\ell}\right).
$$
Thus $G_{P} \leq E$ for all $P \in \Omega_1 \cup \Omega_2$. This implies that every place of $(\F^{\prime})^E$ lying under the places of  $\Omega_1 \cup \Omega_2$ are unramified in the extension  $(\F^{\prime})^E/(\F^{\prime})^G$. If $\ell$ is odd, the ramification index of  the places if $\Omega_3$ in the extension  $(\F^{\prime})^E/(\F^{\prime})^G$ is $(q^2-1)/2\ell$, and we have by Riemann-Hurwitz Formula applied to the extension $(\F^{\prime})^E/(\F^{\prime})^G$
$$
2g((\F^{\prime})^E)-2=-2\ell+(\ell-1),
$$
which leads to a contradiction. If $\ell=2$, we have that the stabilizer of the places in $\Omega_3$ is precisely $E$, and then the extension $(\F^{\prime})^E/(\F^{\prime})^G$ is unramified, which is a contradiction. Therefore, (2) must hold.

Considering now $\F/\fq$, we have a short orbit $\Omega_1$ of size $q+1$ consisting of all $\fq$-rational places. Let $Q$ be a non-rational place with non-trivial stabilizer in $G$. Then $\deg Q >1$. If $\deg Q>2$, then we would have more than two places fixed by $G_Q$, which contradicts (2). Hence $\deg Q=2$, which means that the two places of $\F^{\prime}/\fqc$ fixed by $G$ are above a place of degree two of  $\F/\fq$. Thus, besides $\Omega_1$, we have a unique short orbit $\Omega_2=\{Q\}$, where $\deg Q=2$.

Finally, let $H$ be the only subgroup of $G$ of order $q-1$. For a place $P$ in the  orbit $\Omega_1$, from the first part of the proof, we conclude that $H$ is the stabilizer of $P$ in $G$. Hence $H$ fixes all the places in the short orbits of $G$. Applying the Riemann-Hurwitz Formula to the extension $\F/\F^{H}$ gives $g(\F^{H})=0$.
\end{proof}

\begin{rem}\label{Gfix}
The degree two place $Q$ over $\fq$ splits over $\fqs$ as  $Q_\beta, Q_\gamma$ and
it follows from the proof of Lemma \ref{lemma1}  that  $G$ fixes both $Q_\beta$ and $Q_\gamma$. Moreover $H$ fixes all the places in the short orbits of $G$. 
\end{rem}

The following result will be important in the sequel of the paper. Yet, we point out here that it can be very useful for future references.

\begin{prop}\label{prop1}
Let $\kk$ be a perfect field containing a primitive $n$-th root of unity $\xi$ and $\mathbb{F}$ be a function field over $\kk$. Let $\F=\mathbb{F}(y)$ be a function field defined over $\kk$ by $y^n=h$, where $\car(\kk) \nmid n$ and $h\in \mathbb{F}$, where $h$ is not a $d$-th power in $\mathbb{F}$ for any $d|n$, i.e., $\F$ is a Kummer extension of $\mathbb{F}$. Let $\sigma \in \aut(\F/\kk)$ given by $\sigma(y)=\xi y$ and $\sigma|_{\mathbb{F}}$ is the identity map. If a non-trivial $\tau \in \aut(\F/\kk)$ normalizes $\langle \sigma \rangle$, then there exists $k \in \{1,\ldots,n-1\}$ with $k \nmid n$ and $f \in \mathbb{F}$ such that $\tau(y)=fy^k$, with $f^n=\tau(h)/h^k$. Moreover, if $\tau$ commutes with $\sigma$, then $k=1$.
\end{prop}
\begin{proof}
For $i=0,\ldots,n-1$, let $f_i \in \mathbb{F}$ such that $\tau(y)=\sum\limits_{i=0}^{n-1}f_iy^i$. If $\sigma^\ell \tau =\tau \sigma$, then 
$$
\sum\limits_{i=0}^{n-1}(\xi^{i\ell}-\xi)f_iy^i=0.
$$
From the linear independence of the set $\{1,y,\ldots,y^{n-1}\}$ over $\mathbb{F}$, we conclude that $f_k\neq 0$ for the unique $k \in \{1,\ldots,n-1\}$ such that $k\ell \equiv 1 \mod n$, and $f_i=0$ for $i \neq k$. Thus $\tau(y)=f_ky^k$. Now, on the one hand, $\tau(y^n)=\tau(h)$. On the other hand,
$$
\tau(y^n)=(\tau(y))^n=f_k^n(y^n)^k=f_k^nh^k.
$$
 Hence $f_k^nh^k=\tau(h)$.
\end{proof}

\begin{lem}\label{do k}
Let $\F$ be an algebraic function field over a perfect field $\kk$ having a primitive $n$-th root of unity, for some $n>0$. Let $\F(y)$ be a Kummer extension of $\F$ defined by $y^n=u^k$, where $u \in \F$ and $\gcd(n,k)=1$. Then there exists $z \in \F(y)$ such that $\F(y)=\F(z)$, and $z^n=u$.
\end{lem}
\begin{proof}
Let $r,s \in \Z$ such that $rn+sk=1$. The result is obtained by defining $z:=y^su^r$.
\end{proof}

The main result of the section is the following.

\begin{thm}\label{gencyc}
Let $\F$ be a function field of genus $g$ over $\fq$. Assume that \\
\begin{itemize}
\item $\aut_{\fq}(\F)$ has a subgroup $G$ isomorphic to $\mathbb F_{q^2}^{*}$.\\
\item $g=\frac{(q+1)(q-2)}{2}$.\\
\item $\F$ has exactly $q+1$ rational places.\\
\end{itemize}
Then $\F$ is $\fq$-birationally equivalent to the cyclotomic function field $F(\Lambda_M)$, where $M \in \fq[T]$ is a quadratic monic irreducible polynomial. 
\end{thm}
\begin{proof}
Denote by $H \leq G$ the subgroup of order $q-1$. Since $\F^H$ is rational by Lemma \ref{lemma1}, there exists $v \in \F^H$ such that $\F^H=\fq(v)$, and we have that $\F$ is a Kummer extension of $\fq(v)$ of degree $q-1$. From \cite[Proposition 6.3.1]{St} and Lemma \ref{lemma1}, there exist $r,s_0,\ldots,s_{q-1} \in \mathbb{Z}\backslash\{0\}$ co-primes with $q-1$, where $|r|< q-1$, $|s_i| <q-1$ for all $i$ and $\gcd(q-1,2r+s_0+\cdots+s_{q-1})=1$ , such that $\F=\fq(v,y)$ and
\begin{equation}\label{eqcur1}
y^{q-1}=h(v):=\lambda (v^2+av+b)^r \prod_{\alpha_i \in \fq}(v-\alpha_i)^{s_i},
\end{equation}
where $\lambda, b \in \fq^{*}$, $a \in \fq$ and $T^2+aT+b \in \fq[T]$ is irreducible. Up to an $\fq$-birational map, we may assume that $0<r<q-1$ and
\begin{equation}\label{eqcur2}
h(v)=\frac{\lambda (v^2+av+b)^r}{ \prod\limits_{\alpha_i \in \fq}(v-\alpha_i)^{s_i}},
\end{equation} 
with $0<s_i<q-1$ for all $i$. 
Let $\rho$ be a generator of $G$. Since $G$ is abelian, the restriction of $\rho$ to $\fq(v)$ is an automorphism of $\fq(v)$, which we will denote also by $\rho$. Denote by $Q^{\prime}$ and $P^{\prime}_i$ the places of $\fq(v)$ associated to $v^2+av+b$ and $v-\alpha_i$ for $\alpha_i \in \fq$, respectively, and denote by $P^{\prime}_\infty$ the place at the infinity of $\fq(v)$. Note that these places are exactly the (fully) ramified places of the extension $\F/\fq(v)$. Again, since $G$ is abelian, $\rho$ acts on the set of short orbits of $G$ as it does on the places of $\fq(v)$. Thus, $\rho$ fixes $Q^{\prime}$ and acts transitively on $\{P^{\prime}_0,\ldots,P^{\prime}_{q-1},P^{\prime}_\infty\}$. Let $e$ and $j$ such that  $\rho(P^{\prime}_j)=P^{\prime}_e$, and let $k,l$ such that $\rho(P^{\prime}_l)=P^{\prime}_\infty$ and $\rho(P^{\prime}_\infty)=P^{\prime}_k$. Since $\sigma(\divi(w))=\divi(\sigma(w))$ for all $w \in \F$ and $\sigma \in \aut(\F)$, we obtain
$$
\divi(\rho(h(v)))=rQ^{\prime}+(\sum\limits_{i=0}^{q-1}s_i-2r)P^{\prime}_k-s_lP^{\prime}_\infty-\sum\limits_{i \neq k} \tilde{s}_iP^{\prime}_i,
$$
where $\tilde{s}_i=s_d$ if and only if $ P^{\prime}_i=\rho(P^{\prime}_d)$. Thus
$$
\rho(h(v))=\frac{\alpha (v^2+av+b)^r(v-\alpha_k)^{s_0+\cdots+s_{q-1}-2r}}{\prod\limits_{i \neq k}(v-\alpha_i)^{\tilde{s}_i}}
$$
for some $\alpha \in \fq^*$. 
By the other side, applying Proposition \ref{prop1} we have $\rho(y)=f(v)y$ for some $f(v) \in \fq(v)$ and $h(v)f(v)^{q-1}=h(\rho(v))$. And we obtain
$$
f(v)^{q-1}\cdot \frac{\lambda (v^2+av+b)^r}{ \prod\limits_{\alpha_i \in \fq}(v-\alpha_i)^{s_i}}=\frac{\alpha (v^2+av+b)^r(v-\alpha_k)^{s_0+\cdots+s_{q-1}-2r}}{\prod\limits_{i \neq k}(v-\alpha_i)^{\tilde{s}_i}}.
$$
Thus $\alpha=\lambda$ and
\begin{equation}\label{f(v)}
f(v)^{q-1}=\frac{(v-\alpha_k)^{s_0+\cdots+s_{q-1}-2r+s_k}}{\prod\limits_{i \neq k}(v-\alpha_i)^{\tilde{s}_i-s_i}}.
\end{equation}

Hence, $q-1$ divides both $-2r+s_0+\cdots+s_{q-1}+s_k$ and $\tilde{s}_i-s_i$. In particular, since $\tilde{s}_e=s_j$ and $|s_j-s_e|<q-1$, we can conclude that $s_e=s_j$. Since  $G$ acts transitively on $\{P^{\prime}_0,\ldots,P^{\prime}_{q-1},P^{\prime}_\infty\}$, by taking a suitable power of $\rho$, we obtain that $s_e=s_j$ for arbitrary $e$ and $j$. Thus, we have that
$$
y^{q-1}=\frac{\lambda(v^2+av+b)^r}{(v^{q}-v)^n},
$$
with $0<n<q-1$ and $n$ co-prime with $q-1$.  Moreover, $q-1$ divides $ 2(n-r)$ which is possibly only if $n=r$ or $q-1=2|n-r|$. In the former case, $\F=\fq(v,y)$ with
$$
y^{q-1}= \lambda \left(\frac{v^2+av+b}{v^q-v}\right)^r,
$$
where $\gcd(q-1,r)=1$. By Lemma \ref{do k}, $\F=\fq(v,z)$, with
$$
z^{q-1}= \lambda^{1/r} \cdot \frac{v^2+av+b}{v^q-v}.
$$
From Proposition \ref{pr}, we conclude that $\F$ is $\fq$-isomorphic to $F(\Lambda_M)$, where $M=T^2-\lambda^{1/r}aT+\lambda^{2/r}b$. 

We now proceed to show that the latter case does not hold. For this, assume that $q-1=2|r-n|$. We clearly have that $q$ is odd, and since both $r$ and $n$ are co-prime with $q-1$, we have $q \equiv 1 \mod 4$. Without loss of generality, we may assume that $r>n$. Using \eqref{f(v)} and some results above, we conclude that
$$
f(v)=\xi^j \cdot (v-\alpha_k)^{n-1}, 
$$
where $\xi \in \fq$ is a primitive $(q-1)$-th root of unity and $P^{\prime}_{\alpha_k}$ is the image of $P^{\prime}_{\infty}$ by the restriction of $\rho$ to $\fq(v)$.
Thus,
\begin{equation}\label{rhoy}
\rho(y)=f(v)y=\xi^j \cdot (v-\alpha_k)^{n-1} y.
\end{equation}
However, since $G=\langle \rho \rangle$, we have that $H=\langle \rho^{q+1} \rangle$. In particular, $\rho^{q+1}(y)=\xi y$. On the other hand, using \eqref{rhoy}, we obtain that
$$
\rho^{q+1}(y)=\xi^{2j}\cdot\left(\rho^{q}(v-\alpha_k)\cdot\rho^{q-1}(v-\alpha_k)\cdots \rho(v-\alpha_k) \cdot (v-\alpha_k)\right)^{n-1}y,
$$
which gives that 
\begin{equation}\label{xi}
\xi^{1-2j}=  \left(\rho^{q}(v-\alpha_k)\cdot\rho^{q-1}(v-\alpha_k)\cdots \rho(v-\alpha_k) \cdot (v-\alpha_k)\right)^{n-1}.
\end{equation}
It should be noted that $C:=\rho^{q}(v-\alpha_k)\cdot\rho^{q-1}(v-\alpha_k)\cdots \rho(v-\alpha_k) \cdot (v-\alpha_k)$ is in $\fq$. Indeed, we know that
$$
f(v)=\xi^j \cdot (v-\alpha_k)^{n-1}
$$
is such that 
$$
f(v)^{q-1}=\frac{\rho(h(v))}{h(v)}.
$$
Recall that the restriction of $\rho $ to $\fq(v)$ has order $q+1$. Thus, for $s= (n-1)(q-1)$ we get 
\begin{align*}
C^{s}&=\rho^{q}\left((v-\alpha_k)^{s}\right) \cdot \rho^{q-1}\left((v-\alpha_k)^{s}\right) \cdots \rho\left((v-\alpha_k)^{s}\right) \cdot  (v-\alpha_k)^{s} \\
&= \rho^{q}\left(\frac{\rho(h(v))}{h(v)} \right) \cdot \rho^{q-1}\left(\frac{\rho(h(v))}{h(v)} \right)\cdots \rho\left(\frac{\rho(h(v))}{h(v)} \right)\cdot \frac{\rho(h(v))}{h(v)} \\
&= \left(\frac{h(v)}{\rho^{q}(h(v))} \right) \cdot \left(\frac{\rho^{q}(h(v))}{\rho^{q-1}(h(v))} \right)\cdots \left(\frac{\rho^2(h(v))}{\rho(h(v))} \right)\cdot \frac{\rho(h(v))}{h(v)}\\
& =1.
\end{align*}
This implies that $C^{n-1}$ is a constant in $\fq$, and thus $C$ is in the full constant field of $\F$. Since the full constant field of $\F$ is $\fq$, we conclude that $C \in \fq$. Since $n-1$ is even, we obtain from \eqref{xi} that $\xi$ is a square in $\fq$, which is a contradiction. This finishes the proof.
\end{proof}

\section{The full automorphism group}\label{full}

The aim of this section is to determine the structure of $\aut_{\fqc}(\F)$, where $\F=F(\Lambda_M)$ is the cyclotomic function field  with modulus $M=T^2+aT+b$  irreducible in $\fq[T]$.  Since for $q=2$ we have that $\F$ is a rational function field, which has well known automorphism group, in what follows we always assume $q>2$.

Recall that $\F=\fq(v,y)$, where
$$
y^{q-1}=-\frac{v^2+av+b}{v^q-v}.
$$
The group  $\aut_{\fqc}(\F)$ admits a subgroup $G \cong \fqs^*$ and,  by \cite[Theorem 4.6]{MXY2016}, we also have  $G=\aut_{\fq}(\F)$. To simplify notation from now on we set \\

\begin{itemize}
\item $\F^{\prime}=\fqc \F$. \\
\item $G=\aut_{\fq}(\F)=\langle\rho\rangle$ generated by $\rho$ of order $q^2-1$.\\
\item $H$ the only subgroup of $G$ of order $q-1$, $\F^H=\fq(v)$ and $(\F^{\prime})^H=\fqc(v)$.\\
\item $\bar{G}=\aut_{\fqc}(\F).$\\
\end{itemize}

Assume that $p>2$. Let $\lambda \in \fqs$ be a primitive $2(q-1)$-th root of unity (that is, $\lambda^{q-1}=-1$ and $\lambda^{2}=\xi$ is a primitive $(q-1)$-th root of unity). Then $\bar{G}$ also counts with the automorphism
$$
\mu:(v,y) \mapsto (-v-a,\lambda y).
$$ 
Note that $\mu^2=\rho^{q+1}$ is a generator of $H$. From the proof of Theorem \ref{gencyc}, one can deduce that $\rho(v)=(-(a+\alpha)v-b)/(v-\alpha)$ and $\rho(y)=\xi^{j}(v-\alpha)y$ for some $j \in \{0,\ldots, q-2\}$ and $\alpha \in \fq$. Thus, one can check that $\mu \rho \mu^{-1} \in G$ and we conclude that $ \mu $ normalizes $G$. In fact, the restriction of $\mu \rho \mu^{-1}$ to $\fqc(v)$ is an automorphism of $\fqc(v)$, and it coincides with the restriction of $\rho^{-1}$ to $\fqc(v)$. Thus $\rho(\mu \rho \mu^{-1})$ fixes $\fqc(v)$ elementwise, which means that  $\rho(\mu \rho \mu^{-1}) \in H$ as $\fqc(v)=(\F^{\prime})^H$. Hence $\rho(\mu \rho \mu^{-1})=\rho^{d(q+1)}$ for some $d\geq 1$, and therefore $\mu \rho \mu^{-1} \in \langle\rho\rangle=G$.

 Denote $N:=\langle \rho, \mu \rangle$ the subgroup of $\bar{G}$ generate by $\rho$ and $\mu$. Then 
$$
|N|=2|G|=2(q^2-1).
$$ 

A similar situation holds when $p=2$. As a matter of fact, in this case we have the involution $\omega:(v,y) \mapsto (v+a,y)$, and  $\omega \rho \omega \in G$. Note that $\omega$ is defined over $\fq$, it normailzes $G$ and $N:=\langle \rho, \omega \rangle=G \rtimes \langle \omega\rangle$ has order $2(q^2-1)$.
\begin{rem}
In \cite[Theorem 4.6]{MXY2016}, it was overlooked that $\omega$ does not commute with all elements of $G$. In fact, there is no abelian automorphism group in characteristic $2$ with more than $4g+2$ elements by Theorem \ref{orderab}.
\end{rem}

\begin{prop}\label{normal}
The normalizer of $G$ in $\bar{G}$ is $N$.
\end{prop}
\begin{proof}
Denote such normalizer by $N^\prime$ and recall that from the proof of Lemma \ref{lemma1} the orbits of $G$ in the set of places of $\F^{\prime}$ are 
$\Omega_1=\{P_0,\ldots,P_{q-1},P_\infty\}$, where $P_j$ are all unique extensions of $\fq$-rational places of $\F$, $\Omega_2=\{Q_\beta\}$ and $\Omega_3=\{ Q_\gamma \}$, where $Q_\beta,Q_\gamma $ are the extensions of a place  of degree $2$ of $\F$.

Since every $ \sigma \in N^{\prime}$ induces an automorphism $\tilde{\sigma} \in \aut_{\fqc}((\F^{\prime})^G)$ that acts on the set of places of $(\F^{\prime})^G$ as $\sigma$ does on the set of orbits of $G$, we have that there exists a permutation representation $\varphi:N^{\prime} \rightarrow  \operatorname{Sym}\{Q_\beta, Q_\gamma \}$. 
Let $\sigma \in \ker(\varphi)$ and denote by $\infty, \beta$ and $\gamma$ the places of $(\F^{\prime})^G$ lying under $\Omega_1,\Omega_2$ and $\Omega_3$, respectively. From the fact that $\sigma \in \ker(\varphi)$, we have that $\tilde{\sigma}$ fixes both $\beta$ and $\gamma$. 
Since $\sigma$ must preserve $\Omega_1$, we also have that $\tilde{\sigma}$ fixes $\infty$ as well. Hence, $\tilde{\sigma}$ has $3$ fixed places on a rational function field, which from \cite[Theorem 1]{VM1980} gives that $\tilde{\sigma}$ is the identity. Thus $\sigma \in G$. Therefore, $\ker(\varphi)=G$, which gives $|N^{\prime}|=2(q^2-1)$. Since $N \leq N^{\prime}$, we have $N=N^{\prime}$. 
\end{proof}

Recall that, for each $\fq$-rational place $P_j$ of $\F/\fq$, we have that the conorm of $P_j$ in $\F^{\prime}/\F$ is a single place, which we will also denote by $P_j$, that is, $\con_{\F^{\prime}/\F}(P_j)=P_j$ for $j=\infty,0,1\ldots,q-1$. In the same way, the place $Q$ of $\F/\fq$ over the zero of $v^2+av+b$ has degree $2$, and then $\con_{\F^{\prime}/\F}(Q)=Q_\beta+Q_\gamma$. Since for all divisor $D$ of $\F/\fq$ we have that every basis of the Riemann-Roch space $\mathcal{L}(D)$ over $\fq$ is also a basis of $\mathcal{L}(\con_{\F^{\prime}/\F}(D))$ over $\fqc$ (see e.g. \cite[Theorem 3.6.3]{St}), we conclude that the analogous of \cite[Lemmas 4.2 and 4.3]{MXY2016} hold only replacing $\fq$ by $\fqc$ and the place of degree $2$ by the sum of two places of degree $1$, that is

\begin{equation}\label{L1}
\mathcal{L}((2q-3)P_\infty+Q_\beta+Q_\gamma)=\mathcal{L}((2q-3)P_\infty)+ \fqc \cdot \frac{1}{y} +\fqc \cdot \frac{v}{y},
\end{equation}
and 
\begin{equation}\label{L2}
\mathcal{L}((q-2)P_\infty+Q_\beta+Q_\gamma)=\fqc+ \fqc \cdot \frac{1}{y}.
\end{equation}
 
 Moreover, by \cite[Proposition 4.1]{MXY2016}we also have

\begin{equation}\label{L3}
\mathcal{L}((q-1)P_\infty)=\fqc+ \fqc \cdot v.
\end{equation}

Therefore, when $q\neq 2,3$, the following holds.

\begin{prop}\label{vxing}
Assume that $q>3$. Then the stabilizer of $P_\infty$ in $\bar{G}$ is $\langle \mu \rangle$ if $p>2$, and $\langle \omega \rangle \times H$ if $p=2$.
In particular, the stabilizer of $P_j$ in $\bar{G}$ have size $2(q-1)$  for all $j=\infty,0,\ldots,q-1$.
\end{prop}  
\begin{proof}
The proof will be done for the case $p>2$; the case $p=2$ is analogous. First, $H=\langle\mu^2\rangle$ fixes $P_{\infty}$, and $(\F^{\prime})^{H}=\fqc(v)$. Moreover, $(\F^{\prime})^{\langle \mu\rangle}=\fqc(v^2+av)$, which proves that $\mu$ fixes $P_\infty$.
Now let $\sigma \in \bar{G}_{P_\infty}$.  Since $\sigma(\mathcal{L}((q-1)P_\infty))=\mathcal{L}((q-1)P_\infty)$, From \eqref{L3} we obtain that $\sigma(v)=cv+d$, where $c,d \in \fqc$, $c \neq 0$. In particular, the restriction of $\sigma$ to $\fqc(v)$ gives an automorphism of $\fqc(v)$. Hence, it follows from basic Galois Theory that $\sigma$ normalizes $H$, whence it acts on the set of fixed places of $\F^{\prime}$ as it does on the set of places of $\fqc(v)$ lying under them. For simplicity, we will denote the restriction of $\sigma$ to $\fqc(v)$ also by $\sigma$. Denote by $P^{\prime}_{\infty}$ the place of $\fqc(v)$ lying under $P_{\infty}$. Then $\sigma(P^{\prime}_{\infty})=P^{\prime}_{\infty}$. Since $q>3$, we certainly have that, besides $P^{\prime}_{\infty}$, there are at least $2$ $\fq$-rational places of $\fqc(v)$ sent by $\sigma$  to $\fq$-rational places of $\fqc(v)$. Hence $c,d \in \fq$; furthermore, $\sigma$ preserves the set $\Omega_1$ of $\fq$-rational places of $\F^{\prime}$.

Therefore, arguing exactly as in the proof of \cite[Proposition 4.4]{MXY2016} we obtain $\sigma(y)=\frac{y}{ey+f}$, where $e,f \in \fqc$, $f \neq 0$. Since $\sigma|_{\fqc(v)} \in \aut_{\fqc}(\fqc(v))$, we have from Proposition \ref{prop1} that $\sigma(y)=B(v)y^k$ for some $k \in \{1,\ldots,q-2\}$, $B(v) \in \fqc(v)$. Thus
$$
B(v)y^{k}=\frac{y}{ey+f} \Longrightarrow eB(v)y^{k+1}+fB(v)y^{k}-y=0,
$$
which gives $k=1$, $B(v)=f^{-1}$ and $e=0$. In other words, $\sigma(y)=y/f$. Now, applying $\sigma$ on both sides of 
$$
y^{q-1}=-\frac{v^2+av+b}{v^q-v}
$$
gives 
\begin{equation}\label{auxc1}
(c^qv^q-cv+d^q-d)(v^2+av+b)=f^{q-1}(v^q-v)(c^2v^2+(2d+a)cv+d^2+ad+b).
\end{equation}
Equation \eqref{auxc1} immediately gives that $c=c^2f^{q-1}$, that is, $f^{q-1}=c^{-1}$. Hence, \eqref{auxc1} reads
$$
c^2(v^2+av+b)=(c^2v^2+(2d+a)cv+d^2+ad+b),
$$
implying that 
 $ca=2d+a$;
and $c^2b=d^2+ad+b$. .

Assume that $a=0$. Then $d=0$ and $c^2=1$; in particular, $f \in \fqs$. If $c=1$, then $f \in \fq$, and then $\sigma(v,y)=(v,\xi^j y)$ for some $j\geq 0$, which implies that $\sigma \in H \leq \langle \mu \rangle$. If $c=-1$, then $\sigma(v,y)=(-v,\lambda^\ell y)$, where $\ell >0$ is odd. Thus $\sigma=\mu^\ell \in \langle \mu \rangle$.

Let us assume that $a \neq 0$. Thus $c=\frac{2d}{a}+1$.  Replacing this in $c^2b=d^2+ad+b$, we obtain
$$
\frac{4b}{a}\left(1+\frac{d}{a}\right)=a+d,
$$
which gives $a(4b-a^2)=d(a^2-4b)$. Since $a^2-4b \neq 0$ (as it is the discriminant of $T^2+aT+b$), we conclude that $d=-a$. Therefore $c=-1$ and $f \in \fqs$ such that $f^{q-1}=-1$. Hence, $\sigma \in \langle \mu \rangle$.

The last claim follows from the fact that $\langle\rho \rangle$ acts transitively on $\Omega_1=\{P_0,\ldots,P_{q-1},P_\infty\}$.
\end{proof}

\begin{cor}\label{orbdif}
Assume that $q>3$. If $A \in \{P_j \ | \ j=\infty,0,\ldots,q-1\}$ and $B \in \{Q_\beta,Q_\gamma\}$, then $A$ and $B$ cannot belong to the same orbit of $\bar{G}$.
\end{cor}
\begin{proof}
If $A$ and $B$ are in a same orbit of $\bar{G}$, then their respective stabilizers in $\bar{G}$ have the same order. However, by Proposition \ref{vxing}, $|\bar{G}_A|=2(q-1)$ and we know  that $|\bar{G}_B|\geq q^2-1$ by  Remark \ref{Gfix}.
\end{proof}

We now proceed to show that $\bar{G}$ is tame when $p>2$ and $q \neq 3$. 

\begin{lem}\label{pgrupo}
Assume that $p>2$, $q\neq 3$ and that $\bar{G}$ is non-tame. Then the stabilizers of $Q_\beta$ and $Q_{\gamma}$ are non-tame. 
\end{lem}
\begin{proof}
Let $\sigma \in \bar{G}$ be an automorphism of order $p$ and set $U:=G \cap (\sigma G \sigma^{-1})$. Assume that $|U|>1$. Since $G$ is cyclic, one has $U=\sigma U \sigma^{-1}$. Hence, $\sigma$ acts on the set of short orbits of $U$. However, $U \leq G$ and the short orbits of $G$ are $\Omega_1=\{P_0,\ldots,P_{q-1},P_\infty\}$, $\Omega_2=\{Q_\beta\}$ and $\Omega_3=\{Q_{\gamma}\}$. Since $U$ fixes $Q_{\beta}$, then 
$$
\sigma(Q_{\beta}) \in \Omega_1 \cup \Omega_2 \cup \Omega_3. 
$$
 But from Corollary \ref{orbdif}, we have that  $\sigma(Q_{\beta}) \notin \Omega_1$, which means that the orbit of $Q_{\beta}$ by $\langle \sigma \rangle$ is contained in $\{Q_\beta,Q_{\gamma}\}$. If $\sigma$ does not fix $Q_\beta$, then from the Orbit-stabilizer Theorem we conclude that $p \mid 2$, which contradicts $p>2$. Hence $\sigma$ fixes $Q_{\beta}$. In the same way, $\sigma$ fixes $Q_{\gamma}$. 

Now, assume that $|U|=1$. In this case, we have that $|\bar{G}|>84(g-1)$. Indeed, the set $G \cdot (\sigma G \sigma^{-1}) \subset \bar{G}$ has $(q^2-1)^2$ elements. In particular, $|\bar{G}| >(q^2-1)^2$ as $p$ divides $|\bar{G}|$. From $g=\frac{(q+1)(q-2)}{2}$ we obtain that $|\bar{G}| \leq 84(g-1)$ implies $q\leq 5$. However $|\bar{G}|=2(q^2-1)pk$ for some $k \geq 1$, and since $|\bar{G}| > (q^2-1)^2$, we conclude that $k \geq 3$ when $q=5$. Therefore, we obtain $|\bar{G}|>84(g-1)$. 

We now apply Theorem \ref{orbits}.
Since the $P_j \in \{P_0,\ldots,P_{q-1},P_\infty\}$ and $Q_i \in \{ Q_\beta, Q_\gamma\}$  are not in the same orbit, we have by Proposition \ref{vxing} that (b) and (c) in Theorem \ref{orbits} can be ruled out. If (a) holds, then the proof of Theorem \ref{orbits} gives that the order of the tame stabilizers are $2$; then we can rule out this case, as $|\bar{G}_{P_j}|=2(q-1)$. Thus we have that (d) must hold in Theorem \ref{orbits}. Again by  Proposition \ref{vxing} and Corollary \ref{orbdif}, we conclude that the stabilizers of $Q_\beta$ and $Q_{\gamma}$ are non-tame.
\end{proof}

\begin{prop}\label{tame}
Assume that $p>2$ and $q \neq 3$. Then the full automorphism group  $\bar{G}$ is tame.
\end{prop}
\begin{proof}
Assume that $p$ divides $|\bar{G}|$. Then by Lemma \ref{pgrupo}, we have that the stabilizer $(\bar{G})_{Q_\beta}$ of $Q_{\beta}$ is non-tame.  Denote by $S_p$ the p-Sylow subgroup of $(\bar{G})_{Q_\beta}$ of order $p^r$. We know that $S_p\triangleleft (\bar{G})_{Q_\beta}$ and $(\bar{G})_{Q_\beta}/S_p$ is cyclic with order not divisible by $p$, see Theorem \ref{stab}. 

Denote by $\tilde{g}$ the genus of the field $(\F^{\prime})^{S_p}$ fixed by $S_p$. By the Riemann-Hurwitz formula, we obtain 
$$
\tilde{g} < \frac{g}{p^r}+1-\frac{1}{p^r}.
$$

Since $(\F^{\prime})^{S_p}$ has an automorphism group isomorphic to $(\bar{G})_{Q_{\beta}}/S_p$ and $G \leq (\bar{G})_{Q_{\beta}}$, we have that there exists $\tilde{G} \cong G$ such that $\tilde{G} \leq \aut_{\fqc}((\F^{\prime})^{S_p})$.
First, let us suppose that $\tilde{g} >1$. Then by Theorem \ref{orderab}, we have that $|\tilde{G}| \leq 4\tilde{g}+4$ as $\tilde{G}$ is abelian. However, $|\tilde{G}|=q^2-1>2g+\sqrt{2g}$, which leads us to a contradiction, since $g>p^r(\tilde{g}-1)$+1. 

The case $\tilde{g}=1$ can be easily ruled out by \cite[Theorem 11.94]{HKT} since $\tilde{G}$ fixes a place of $(\F^{\prime})^{S_p}$ and $q^2-1 \notin \{2,4,6,12\}$. 

Finally, let us assume that $\tilde{g}=0$. Note that since $\rho$ normalizes $S_p$, we have that $\rho^{q+1}$ does it as well. In particular, there exists $\tilde{H}\leq \aut_{\fq}((\F^{\prime})^{S_p})$ such that $\tilde{H} \cong H$  , and $\tilde{H}$ acts on the places of $(\F^{\prime})^{S_p}$ as $H$ does on the set of orbits of $S_p$ in $\F^{\prime}$. Since $S_p$ fixes $Q_{\beta}$, we have that $Q_{\gamma}$ is in an orbit of $S_p$ distinct from the orbits of $P_0$ and $Q_{\beta}$.  Denote by $\tilde{P}$, $\tilde{Q_{\beta}}$ and $\tilde{Q_{\gamma}}$ the places of $(\F^{\prime})^{S_p}$ lying under $P_0$, $Q_{\beta}$ and $Q_{\gamma}$, respectively. Since $H$ fixes $P_0$, $Q_{\beta}$ and $Q_{\gamma}$, we must have then that $\tilde{H}$ fixes $\tilde{P}$, $\tilde{Q_{\beta}}$ and $\tilde{Q_{\gamma}}$. But this is a contradiction, since a cyclic automorphism group in a rational function field fixes at most two places, see \cite[Theorem 1]{VM1980}. This finishes the proof.
\end{proof}

The main result of the section is the following.

\begin{thm}\label{mainfull}
Assume that $p>2$.  Then
\begin{itemize}
 \item[(i)] If $q\neq 3$, the full automorphism group $\aut_{\fqc}(\F)$ has order $2(q^2-1)$. More precisely, $\aut_{\fq}(\F) \cong \Z_{q^2-1}$ is a normal cyclic subgroup of $\aut_{\fqc}(\F)$, and 
$$
\frac{\aut_{\fqc}(\F)}{\aut_{\fq}(\F)} \cong \Z_2.
$$ 
\item[(ii)] If $q=3$, the full automorphism group $\aut_{\fqc}(\F)$ has order $6(q^2-1)=48$. More precisely, $\F^{\prime}$ is a hyperelliptic function field with hyperelliptic involution $\iota=\rho^{q+1}$, and 
$$
\frac{\aut_{\fqc}(\F)}{\langle \iota \rangle} \cong  \PGL(2,3).
$$
\end{itemize}
\end{thm}
Moreover, in both cases  $\aut_{\fqc}(\F)=\aut_{\fqs}(\F)$.
\begin{proof}
First assume that $q \neq 3$. We aim to show that $G$ is a normal subgroup of $\bar{G}$. The result then will follow immediately from Proposition \ref{normal}. Let $\delta \in \bar{G}\backslash G$. We claim that there exists a non-trivial $K \leq G$ such that $\delta K\delta^{-1}=K$. Indeed, assume that $\delta G\delta^{-1} \cap G$ is trivial. Then arguing as in the proof of Lemma \ref{pgrupo}, we conclude that $\bar{G}$ has a subgroup of order $> (q^2-1)^2$, as $\delta \notin (\delta G\delta^{-1}) \cdot G$. However, $\bar{G}$ is tame by Proposition \ref{tame}, and then by Theorem \ref{orbits} we have that $|\bar{G}| \leq 84(g-1)$. Thus we obtain $q \leq 5$. We now proceed to exclude case $q=5$.

If $q=5$, we have $|\bar{G}|>(q^2-1)^2=576$ and $|N|=2(q^2-1)=48$ divides $|\bar{G}|$. Hence for some $k>0$, we have $576<16k \leq 672$, which gives $k=13$ or $k=14$. Thus we conclude that $|\bar{G}|=624$ or $|\bar{G}|=672=84(g-1)$; the former case implies that $\bar{G}$ has a subgroup of order $13$, which is forbidden by \cite[Theorem 1]{Homma1980}. If $|\bar{G}|=672$,  from Sylow Theory we obtain that the normalizer of a $7$-Sylow subgroup $S_7$ in $\bar{G}$ has an automorphism of order $3$. Moreover, by the Riemann-Hurwitz Formula, the fixed field $(\F^{\prime})^{S_7}$ has genus $0$, and $S_7$ fixes $5$ places of $\F^{\prime}$. Hence, there is a subgroup $J \leq \bar{G}$ of order $3$ acting on the set of fixed places of $S_7$. Let $K$ be the subgroup of order $3$ of $G$. Since $K$ is a $3$-Sylow subgroup, there exists $\sigma \in \bar{G}$ such that $K=\sigma J \sigma^{-1}$. Note that $J$ must fix a place $P$ that is fixed by $S_7$. Thus $K$ fixes $\sigma(P)$, which gives that $\sigma(P) \in \{Q_\beta, Q_\gamma\}$. Furthermore, $\sigma S_7 \sigma^{-1}$ fixes $\sigma(P)$ as well. In particular, the stabilizer $(\bar{G})_{\sigma(P)}$ of $\sigma(P)$ contains $G$ (which has order $24$) and an automorphism of order $7$. Hence $|(\bar{G})_{\sigma(P)}|>168>40=4g+4$, which is a contradiction since $(\bar{G})_{\sigma(P)}$ is cyclic (Theorem \ref{orderab} and Theorem \ref{stab}).

Therefore we conclude $K:=\delta G\delta^{-1} \cap G$ is non-trivial. Since $G$ is cyclic, we have that $\delta K\delta^{-1}=K$. In particular, $\delta$ acts on the set of short orbits of $K$.

Now, note that, since a tame stabilizer of a place is abelian by Theorem \ref{stab}, from Proposition \ref{tame} we obtain that $\bar{G}_{Q_\beta}=\bar{G}_{Q_\gamma}=G$. Furthermore, the short orbits of $K$ are contained in $\{P_{\infty},P_0,\ldots,P_{q-1},Q_{\beta},Q_{\gamma}\}$. Since $G$ fixes elementwise both $Q_{\beta}$ and $Q_{\gamma}$, and there is no automorphism of $\F^{\prime}$ sending $Q_\beta$ and $Q_\gamma$ in one of the $P_j$, we have that 
$$
\delta:\{Q_\beta,Q_\gamma\} \rightarrow \{Q_\beta,Q_\gamma\}.
$$
Since $\bar{G}_{Q_\beta}=\bar{G}_{Q_\gamma}=G$ and $\delta \notin G$, we obtain that $\delta(Q_{\beta})=Q_{\gamma}$ and $\delta^2 \in G$. Finally, 
$$
(\delta \rho \delta^{-1})(Q_{\beta})=(\delta \rho)(Q_{\gamma})=\delta(Q_{\gamma})=Q_\beta,
$$ 
and so $\delta \rho \delta^{-1} \in \bar{G}_{Q_\beta}=G$, that is, $\delta \in N$. Therefore $N=\bar{G}$.

Assume now that $q=3$. In this case, one can check that, up to $\mathbb{F}_3$-isomorphism, $\F=\mathbb{F}_3(v,y)$ with 
$$
y^2=\frac{v^2+1}{v^3-v}.
$$
Let $i \in \mathbb{F}_9$ such that $i^2=-1$. Then one can check that 
$$
\varepsilon:(v,y) \mapsto \left(-\frac{v+i}{v-i},\frac{i(1-i)(v^2-1)y}{v+i} \right)
$$
is an automorphism of order $3$ of $\F^{\prime}$ defined over $\mathbb{F}_9$.
Hence $\aut_{\fqs}(\F)$ has order at least $48$. From \cite[Theorem 11.127]{HKT}, the function field $\F^{\prime}$ cannot have more automorphisms, which means that $|\bar{G}|=48$ . The subgroup $H$ is generated by the central hyperelliptic involution $\iota:=\rho^4$. Thus $\bar{G}/H$ is isomorphic to a subgroup of $\aut(\fqc(v))$ of order $24$ defined over $\fqs$, that is, $\bar{G}/H$ is isomorphic to a subgroup of order $24$ of $\PGL(2,9)$. Hence, from \cite[Theorem 3]{VM1980}, we have that either $\bar{G}/H \cong \PGL(2,3)$ or $\bar{G}/H$ is isomorphic to a semidirect product of a group of order $3$ with an abelian group of order $8$.  If the later case holds, then since there is a copy of $G/H$ in $\aut(\fqc(v))$ fixing the places lying under $Q_\beta$ and $Q_\gamma$, we have from \cite[Theorem 1]{VM1980} that an automorphism of order $8$ must fix  these places.  In particular, there exists an automorphism subgroup of order $16$ fixing  $Q_\beta$. From Theorem \ref{stab} this group is abelian, which is a contradiction to Theorem \ref{orderab}. Therefore, we must have $\bar{G}/H \cong \PGL(2,3)$.
\end{proof}

\begin{rem}
Notice that we conclude from Theorem \ref{mainfull} and the discussion in the beginning of this section that the full automorphism group of $F(\Lambda_M)$ when $M$ is monic, quadratic and irreducible is non-tame if, and only if, $p=2$ or $q=3$. We also highlight that, although Theorem \ref{mainfull} do not comprise the case $p=2$, we have that Proposition \ref{normal}, Proposition \ref{vxing} and Corollary \ref{orbdif} hold for $p=2$. 
\end{rem}

\section*{Acknowledgements}

The research of Nazar Arakelian was partially supported by grant 2023/03547-2, São Paulo Research Foundation (FAPESP). The research of Luciane Quoos  was partially supported by CNPq (Bolsa de Produtividade 302727/2019-1, FAPERJ/SEI-260003/002364/2021).

\bibliographystyle{abbrv}


\bibliography{bibcyclo}

\end{document}